\documentclass[11pt]{article}
\usepackage[utf8]{inputenc}
\usepackage[T1]{fontenc}
\usepackage{lmodern}
\usepackage[a4paper]{geometry}
\usepackage[english]{babel}
\usepackage{pifont}
\usepackage{extsizes}
\usepackage{xcolor}
\usepackage{fancybox}
\usepackage{graphicx}
\usepackage{sectsty}
\usepackage{enumerate}
\usepackage{multicol}
\usepackage[english]{varioref}
\usepackage[fleqn]{amsmath}
\usepackage{mathtools}
\usepackage{amssymb}
\usepackage{amsmath}
\usepackage{mathrsfs}
\usepackage{bm}
\usepackage{amsthm}
\usepackage{blkarray}
\usepackage{subfigure}
\usepackage{lipsum}
\usepackage{stmaryrd}
\usepackage[babel=true]{csquotes}
\theoremstyle{break}
\newtheorem{Th}{Theorem}

\newtheorem{Rk}{Remark}

\newtheorem{Ex}{Example}
\newtheorem{Le}{Lemma}
\newtheorem{Cor}{Corollary}

\newcommand{\PP}{\mathbb{P}}

\newcommand{\NN}{\mathbb{N}}

\newcommand{\ind}[1]{\mathbf{1}_{\{#1\}}\,}
\newcommand\blfootnote[1]{%
  \begingroup
  \renewcommand\thefootnote{}\footnote{#1}%
  \addtocounter{footnote}{-1}%
  \endgroup
}

\newenvironment{prooft}[1]{\vskip 2mm\noindent {\bf Proof of #1.}}
                    {\hfill $\square$ \vskip 2mm \noindent}
\date{}  
\begin{document}

\begin{center}
\huge{Convergence of weighted ergodic averages}
\end{center}

\begin{center}
\large{Ahmad Darwiche \& Dominique Schneider}\footnote{Univ. Littoral C\^ote d'Opale, UR 2597, LMPA, Laboratoire de Math\'ematiques Pures et Appliqu\'ees Joseph Liouville, F-62100 Calais, France. Mails: ahmad.darwiche@univ-littoral.fr \& dominique.schneider@univ-littoral.fr}
\end{center}

\begin{abstract}
    Let $(X, \mathcal{A},\mu)$ be a probability space and let $T$ be a contraction on $L^2(\mu)$.  We provide suitable conditions over sequences $(w_k)$, $(u_k)$ and $(A_k)$ in such a way that the weighted ergodic limit  $\lim\limits_{N\rightarrow\infty}\frac{1}{A_N}\sum_{k=0}^{N-1} w_k T^{u_k}(f)=0$ $\mu$-a.e. for any function $f$ in $L^2(\mu)$. As a consequence of our main theorems, we also deal with the so-called one-sided weighted ergodic Hilbert transforms.
\begin{center}
\textbf{R\'esum\'e}
\end{center}
    Soient $(X, \mathcal{A},\mu)$ un espace de probabilit\'e et  $T$ une contraction agissant sur $L^2(\mu)$. Nous donnons des conditions sur les suites $(w_k)$, $(u_k)$ et $(A_k)$ de sorte que la limite des moyennes ergodiques pond\'er\'ees $\lim\limits_{N\rightarrow\infty}\frac{1}{A_N}\sum_{k=0}^{N-1} w_k T^{u_k}(f)=0$ $\mu$-p.p. pour toute $f$ dans $L^2(\mu)$. Comme
cons\'equence de nos principaux r\'esultats, nous \'etudions \'egalement des propri\'et\'es de convergence ponctuelle de la transform\'ee de Hilbert unilat\'erale pond\'er\'ee.
\end{abstract}

\section{Introduction}

 Let $(X, \mathcal{A},\mu)$ be a probability space and let $T:X\rightarrow X$ be an invertible transformation preserving $\mu$, i.e. $\mu(T^{-1}A)=\mu(A)$ for any $A\in \mathcal{A}$. In what follows, we write  $(T^kf)(x)=f(T^kx)$ for any function $f:X\rightarrow \mathbb{R}$, $x\in X$ and $k\geq 1$. The study of averages over iterations of $T$ started in 1930s with the classical Birkhoff's Ergodic Theorem \cite{Birkhoff656}. Wiener and Wintner  \cite{Wiener1941HarmonicAA} generalized this result by showing that, for any $f\in L^1(\mu)$, $\mu-$almost everywhere ($\mu$-a.e.), for any $\theta \in \mathbb{R}$,
 \[\lim\limits_{N\rightarrow\infty}\frac{1}{N}\sum_{k=0}^{N-1} e^{2i \pi k\theta} T^{k}(f)=0.\] Lesigne \cite{lesigne_1990} went on to prove the same property when the term $e^{2i\pi k\theta}$ is replaced by any weight of the form $w_k= e^{2i \pi P(k)}$, where $P$ is a polynomial. In this way, many results on the convergence of averages of the form
 \begin{align}{\label{average1}}
   \frac{1}{N}\sum_{k=0}^{N-1} w_k T^{u_k}(f)
\end{align}
were established for various weights $(w_k)$ and for various sequences of integers  $(u_k)$.

In this paper, we consider a more general problem in the following sense. Let $(w_k)$, $(u_k)$ and $(A_k)$ be sequences of complex numbers, integers and real numbers, respectively. We provide suitable conditions over these sequences to ensure that the weighted ergodic average \begin{align*}
  \frac{1}{A_N}\sum_{k=0}^{N-1} w_k T^{u_k}(f),  
\end{align*}
converges to $0$ $\mu$-a.e., where $T$ is a contraction on $L^2(\mu)$ and where $f\in L^2(\mu)$. We also investigate the everywhere convergence of the so-called one-sided weighted ergodic Hilbert transforms, namely
\[\sum_{k\geq 1} \frac{w_k}{A_k}T^{u_k}(f).\]

Before stating our theorems, we first give some notation and conditions.  For any $M<N$, we let \[V_{M,N}(\theta)=\sum_{k=M}^{N-1}w_ke^{2i\pi \theta u_k}.\] When $M=0$, we take the convention $V_{0,N}(\theta)=V_N(\theta)$. In what follows, we denote by $C$ a generic positive constant which may differ from line to line. We will mainly consider the two following conditions, referred to as Condition $(H1)$ and Condition $(H2)$.

 \textit{Condition $(H1)$:} we say that the sequences $(w_k)$ and $(u_k)$ satisfy condition $(H1)$ if there exist $\delta\geq 0$, $\frac{1}{2}\leq \alpha<1$ and $\beta\in \mathbb{R}$, with $\delta+\alpha<1$, such that for any $M<N$,
\begin{align}{\label{CTh1}}
    \sup_{\theta \in \mathbb{R}}|V_{M,N}(\theta)|\leq C N^\delta(N-M)^{\alpha}\log^\beta N.
\end{align}

 \textit{Condition $(H2)$:}  we say that the sequences $(w_k)$ and $(u_k)$ satisfy condition $(H2)$ if there exist $\frac{1}{2}\leq \alpha\leq 1$ and $\beta \in \mathbb{R}$ such that, for any $M<N$,
\begin{align}{\label{CTh2}}
    \sup_{\theta \in \mathbb{R}}|V_{N}(\theta)|\leq C N^{\alpha}\log^\beta N.
\end{align}

Notice that, if $(H1)$ holds with $\delta=0$, then $(H2)$ also holds. Moreover, when $(w_k)$ is a $q-$multiplicative sequence (see Section 4 in \cite{Fan1} for a definition), for some integer $q\geq 2$, it was proved that \eqref{CTh1} holds with $\beta=\delta=0$ if and only if \eqref{CTh2} holds with $\beta=0$ (see Theorem 4 in \cite{Fan1}). This type of conditions has already been considered in the classical literature, see e.g. \cite{Schneider,Fan1}.

In \cite{Schneider}, Durand and Schneider proved that, if  \eqref{CTh2} holds with $\beta=0$, then for any dynamical system  $(X, \mathcal{A},\mu,T)$, for any $\epsilon >0$, $f \in L^{2+\epsilon}(\mu)$, $\beta' >\frac{\alpha+2}{3}$,
\begin{align}{\label{limit1}}
 \lim\limits_{N\rightarrow\infty}\frac{1}{N^{\beta'}}\sum_{k=0}^{N-1} w_k T^{u_k}(f)=0 \quad \mu-a.e..
\end{align}


\bigskip
More recently, Fan \cite{Fan1} investigated weighted ergodic averages under Condition $(H1)$, with $\delta=\beta=0$. He obtained that for any dynamical system $(X, \mathcal{A},\mu,T)$, and for any $f \in L^2(\mu)$,  
\begin{align*}\lim\limits_{N\rightarrow\infty}\frac{1}{N^\alpha \log^2N \phi(\log N)}\sum_{k=0}^{N-1} w_k T^{u_k}(f)=0\quad \mu-a.e..
\end{align*}
In the above equation, the function $\phi: \mathbb{R}_+\to \mathbb{R}_+$ is assumed to be increasing and to satisfy the following properties:  $\phi(x)\leq C \phi(2x)$  for any $x\in \mathbb{R}_+$, with $C>1$, and $\sum_{n=0}^\infty\frac{1}{n\phi(n)}<\infty$.

 Our first theorem improves Fan's result as follows.
 

\begin{Th}{\label{Th1}}
Assume that Condition $(H1)$ holds. Let $H$ be such that $H>\frac{1}{2}+\beta$ if $\alpha \neq \frac{1}{2}$, or such that $H>\frac{3}{2}+\beta$ if $\alpha= \frac{1}{2}$, where $\alpha,\beta$ are as in Condition $(H1)$. Then for any contraction $T$ on $L^2(\mu)$, and for any $f \in L^2(\mu)$, we have
\begin{align*}
\lim\limits_{N\rightarrow\infty}\frac{1}{N^{\alpha+\delta}\log^HN}\sum_{k=0}^{N-1} w_k T^{u_k}(f)=0 \quad \mu-a.e.
\end{align*}
and
\begin{align*}
    \left|\left|\sup_{N>1}\left|\frac{1}{N^{\alpha+\delta}\log^HN}\sum_{k=0}^{N-1} w_k T^{u_k}(f)\right| \, \right|\right|_{2,\mu}\leq C\left|\left|f\right|\right|_{2,\mu}.
\end{align*}
\end{Th}
In the above equation, $||\cdot ||_{2,\mu}$ stands for the standard norm in $L^2(\mu)$.  As a corollary, we obtain the following result dealing with the convergence of the one-sided weighted ergodic Hilbert transform $\sum_{k>1} \frac{w_k}{k^{\alpha+\delta}\log^Hk} T^{u_k}(f)$.
 \begin{Cor}{\label{cor1}}
Assume that Condition $(H1)$ holds. Let $H$ be such that $H>1+\beta$ if $\alpha \neq \frac{1}{2}$, or such that $H>\frac{3}{2}+\beta$ if $\alpha= \frac{1}{2}$, where $\alpha,\beta$ are as in Condition $(H1)$. Then for any contraction $T$ on $L^2(\mu)$, and for any $f \in L^2(\mu)$, the following series \[\sum_{k>1} \frac{w_k}{k^{\alpha+\delta}\log^Hk} T^{u_k}(f)\]
exists $\mu$-a.e., and
\[
\left\lvert\left\lvert \sup_{N>1} \left\lvert\sum_{k=2}^{N} \frac{w_k}{k^{\alpha+\delta}\log^H k} T^{u_k}(f) \right\rvert \,\right\rvert\right\rvert_{2,\mu} \leq C ||f||_{2,\mu}.
\]
\end{Cor}
 
 The following theorem deals with ergodic averages under Condition $(H2)$, with $\alpha<1$.
\begin{Th}{\label{Th2}}
Let $(w_k)$ be a bounded sequence of complex numbers and let $(u_n)\subset{\NN}$ be a sequence of integers such that Condition $(H2)$ holds, with $\alpha<1$. Then for any  contraction $T$ on $L^2(\mu)$, and for any $f \in L^2(\mu)$, $H>\frac{3}{2}+\beta$, we have
\begin{align*}
 \lim\limits_{N\rightarrow\infty}\frac{1}{N^{\frac{\alpha+1}{2}}\log^HN}\sum_{k=0}^{N-1} w_k T^{u_k}(f)=0 \quad \mu-a.e.
\end{align*}
and
\begin{align*}
    \left|\left|\sup_{N>1}\left|\frac{1}{N^{\frac{\alpha+1}{2}}\log^HN}\sum_{k=0}^{N-1} w_k T^{u_k}(f)\right| \, \right|\right|_{2,\mu}\leq C\left|\left|f\right|\right|_{2,\mu}.
\end{align*}
\end{Th}
Notice that our result generalizes \cite{Schneider} into two ways: first, in our theorem, the term appearing in the ratio is $N^{\frac{\alpha+1}2}$, whereas in  \eqref{limit1} the term which is considered is $N^{\beta'}$, with $\beta'>\frac{\alpha+2}{3}$; secondly we only assume that $f\in L^2(\mu)$ whereas \eqref{limit1} holds only for functions $f\in L^{2+\epsilon}(\mu)$.
\begin{Cor}{\label{cor2}}
Under the same assumption as in Theorem \ref{Th2}, for any contraction $T$ on $L^2(\mu)$, for any $f \in L^2(\mu)$ and for any $H>\frac{3}{2}+\beta$, the series
\[\sum_{k\geq 1} \frac{w_k}{k^{\frac{\alpha+1}{2}}\log^H k} T^{u_k}(f)\]
exists $\mu$-a.e., and
\[\left\lvert\left\lvert \sup_{N>1} \left\lvert\sum_{k=2}^{N} \frac{w_k}{k^{\frac{\alpha+1}{2}}\log^H k} T^{u_k}(f) \right\rvert \,\right\rvert\right\rvert_{2,\mu} \leq C ||f||_{2,\mu}.\]
\end{Cor}
The following theorem deals with the case  $\alpha=1$.
\begin{Th}{\label{Th3}}
Let $(w_k)$ be a bounded sequence of complex numbers and let $(u_n)\subset{\NN}$ be a sequence of integers. Suppose that, for some $\beta>0$,  
\begin{align}{\label{CTh3}}
    \sup_{\theta \in \mathbb{R}}\left|V_N(\theta)\right|\leq C \frac{N}{\log^\beta N}.  
\end{align}
\begin{enumerate}
    \item If $\beta>1$, then for any contraction $T$ on $L^2(\mu)$ and for any $f \in L^2(\mu)$, we have  
\[
\lim\limits_{N\rightarrow\infty}\frac{1}{N}\sum_{k=0}^{N-1} w_k T^{u_k}(f)=0\quad \mu-a.e. \quad \text{and} \quad
    \left|\left|\sup_{N\geq1}\left|\frac{1}{N}\sum_{k=0}^{N-1} w_k T^{u_k}(f)\right| \, \right|\right|_{2,\mu}\leq C\left|\left|f\right|\right|_{2,\mu}.
\]
\item If $\frac{1}{2}<\beta\leq1$, then for any contraction $T$ on $L^2(\mu)$, for any $f \in L^2(\mu)$ such that there exists $ \tau>4+\frac{2}{2\beta-1}$ with $\int_X \lvert f\rvert^{\frac{2}{2\beta-1}}\log^\tau(1+\lvert f\rvert)\mathrm{d}\mu<\infty$, we have  
\begin{align*}
\lim\limits_{N\rightarrow\infty}\frac{1}{N}\sum_{k=0}^{N-1} w_k T^{u_k}(f)=0 \quad \mu-a.e..
\end{align*}
\end{enumerate}
\end{Th}

Similar results were established in the particular case where the weight is a M\"{o}bius sequence $(\mu(k))$. Recall that such a sequence is defined as  $\mu(1)=1$, $\mu(k)=(-1)^j$ if $k$ is a product of $j$ distinct primes, and  $\mu(k)=0$ otherwise. By checking a condition similar to \eqref{CTh3}, Abdalaoui \textit{et al} \cite{Abdalaoui} proved that
\[\lim\limits_{N\rightarrow\infty}\frac{1}{N}\sum_{k=0}^{N-1} \mu(k) T^{k}(f)=0 \quad \mu-a.e..\]
Eisner \cite{Eisner2015APV} showed that this result remains true if we replace $T^k(f)$ by $T^{p(k)}(f)$, where  $p$ is a polynomial.
In 2017, Fan \cite{Fan1} extended these results to any  Bourgain's sequence (see Section 2 for a precise definition in \cite{Fan1}), and for any bounded sequence $(w_k)$ satisfying Equation \eqref{CTh3}, with $\beta>\frac{1}{2}$. Our Theorem \ref{Th3} deals with a more general problem since we consider a general sequence of integers $(u_k)$, and not necessarily Bourgain's sequences.


By adapting the proof of Theorem \ref{Th3}, we can show that if we have $\sup_{\theta \in \mathbb{R}}|V_{M,N}(\theta)|\leq C \frac{N-M}{\log^\beta N}$ for any $M<N$, then the second assertion of Theorem \ref{Th3} remains true for any function $f\in L^2(\mu)$, which do not necessarily satisfy the condition $\int_X \lvert f\rvert^{\frac{2}{2\beta-1}}\log^\tau(1+\lvert f\rvert)\mathrm{d}\mu<\infty$.  
\begin{Cor}{\label{cor3}}
Assume that \eqref{CTh3} holds.\\
\begin{enumerate}
    \item If $\beta>1$, then for any contraction $T$ on $L^2(\mu)$ and for any $f \in L^2(\mu)$, we have
\begin{align*} \sum_{k\geq 1} \frac{w_k}{k} T^{u_k}(f)\quad \text{exists} ~ \mu-a.e. \quad \text{and} \quad \left\lvert\left\lvert \sup_{N\geq1} \left\lvert\sum_{k=1}^{N} \frac{w_k}{k} T^{u_k}(f) \right\rvert \,\right\rvert\right\rvert_{2,\mu} \leq C ||f||_{2,\mu}.
\end{align*}
\item If $\frac{1}{2}<\beta\leq1$, then for any contraction $T$ on $L^2(\mu)$, for any $H>1-\beta$, for any $f \in L^2(\mu)$ such that there exists $ \tau>4+\frac{2}{2\beta-1}$ with $\int_X \lvert f\rvert^{\frac{2}{2\beta-1}}\log^\tau(1+\lvert f\rvert)\mathrm{d}\mu<\infty$, we have
\[\sum_{k>1} \frac{w_k}{k \log^{H} k} T^{u_k}(f)\]
exists $\mu$-a.e..
\end{enumerate}

\end{Cor}
\bigskip
The rest of this paper is organized as follows. In Section $2$, we recall some known results, which will be used in the proofs of our main theorems, including spectral theorem, M\'oricz's lemma and  van der Corput's theorem. In Section 3, we prove Theorems \ref{Th1}, \ref{Th2} and \ref{Th3}. In Section $4$, we discuss convergences of ergodic weighted averages with respect to the harmonic density. In the last section, we provide several examples illustrating our main results. The type of questions that we deal with has already been investigated and many methods have been provided in this way (see \cite{Weber2000} for more details). However, our methods are simpler and provide new examples.      

\section{Preliminaries}
\label{sec:preliminaries}
In this section, we recall some known results. The first one, referred to as the spectral lemma, reduces the problem of evaluating norms to Fourier analysis questions. Let $T$ be a contraction in a Hilbert space $\mathcal{H}$, that is, a linear operator such that $||T(f)||\leq ||f||$ for each $f \in \mathcal{H}$, and put
\begin{align*}
    P_n(f)=\langle T^n(f),f \rangle \quad \text{for} \quad n\geq0 \quad \text{and} \quad P_n(f)=\overline{P_{-n}(f)} \quad \text{for} \quad n<0.
\end{align*}
 The sequence $(P_n(f))_{n\in \mathbb{Z}}$ is non-negative definite. By Herglotz's theorem, there exists a finite positive measure $\mu_f$ on $\mathcal{B}(\mathbb{R}/\mathbb{Z})$ (called the spectral measure of $f$) such that for all $n\geq 0$, we have
 \[\langle T^n(f),f \rangle =\int_{\mathbb{R}/\mathbb{Z}} \exp(2i\pi n t) \mu_f(\mathrm{d}t). \]
The spectral lemma can be stated as follows.
\begin{Le}\texttt{(Spectral lemma)}. Let $T$ be a contraction in a Hilbert space $\mathcal{H}$ and let $f \in \mathcal{H}$. Let $P(x)=\sum_{k=0}^Na_kx^k$ be a complex polynomial of degree $N\geq 0$. Then we have
\[ ||P(T)f||^2\leq \int_{\mathbb{R}/\mathbb{Z}}|P(\exp(2i\pi t)|^2 \mu_{f}(\mathrm{d}t).\]
\end{Le}

The second result we recall is due to M\'oricz. A direct consequence of Theorems 1 and 3 in \cite{Moricz1976} is the following lemma.    
\begin{Le}\texttt{(M\'{o}ricz's lemma)}. Let $(N_k)$ be an increasing sequence of integers and let $(X_n)_{n\geq 1}$ be a sequence of random variables on a probability space
$(X,\mathcal{A},\mathbb{P})$. Assume that there exist $\lambda\geq 1$ and a non-negative function $g:\mathbb{N}\times \mathbb{N} \longmapsto \mathbb{R}^+$, such that for any $j\in \mathbb{N}^*$, $m,n \in \llbracket N_j,N_{j+1}\rrbracket$ with $m<n$
\begin{enumerate}
    \item $\mathbb{E}(|\sum_{k=m+1}^{n}X_k|^2)\leq Cg^{\lambda}(m,n)$,
    \item  $g(m,j)+g(j,n)\leq g(m,n)$ for all $j \in \llbracket m,n \rrbracket $.
\end{enumerate}
    Then, if $\lambda>1$
   \[\mathbb{E}\left(\sup_{j=m+1}^{n}\left\lvert\sum_{k=m+1}^{j}X_k\right\rvert^2\right)\leq C(\lambda)g^{\lambda}(m,n) \qquad \text{for any} ~ m,n \in \llbracket N_j,N_{j+1}\rrbracket \, \text{with}~ m<n,\]
    where $C(\lambda)$ is a constant only depending on $\lambda$. \\
    If $\lambda=1$,
     \[\mathbb{E}\left(\sup_{j=m+1}^{n}\left\lvert\sum_{k=m+1}^{j}X_k\right\rvert^2\right)\leq C\log^2\left(2(n-m)\right)g(m,n) \quad \text{for any}~ m,n \in \llbracket N_j,N_{j+1}\rrbracket \, \text{with}~ m<n,\]
     where $C$ is an absolute constant.
\end{Le}

The following lemmas are due to van der Corput (see Theorems 3 and 4 in \cite{CorputNeueZA}) and will be used in our examples.
\begin{Le}
Let $a$ and $b$ be two integers, with $a<b$, and let $f$ be a function on $[a,b]$. Assume that the second derivative is such that  $-f''(x)\geq \rho $ for any $x \in [a,b]$, and for some $\rho >0$. Then
\[\left\lvert\sum_{k=a}^b e^{2i\pi f(k)} \right\rvert \leq (|f'(b)-f'(a)|+2)\left(\frac{4}{\sqrt{\rho}}+3\right).\]
\end{Le}
\begin{Le}
Let $n\geq 2$ be an integer and put $K=2^n$. Suppose that $a\leq b\leq a+N$ and that $f:[a,b]\to
\mathbb{R} $ has continuous $n$th derivative that satisfies the inequality $0<\lambda\leq \lvert f^{(n)}(x)\rvert\leq h\lambda$ for any $x \in [a,b]$. Then
\begin{align*}
    \left\lvert \sum_{k=a}^b e^{2i\pi f(k)} \right\rvert \leq C h N \left(\lambda^{\frac{1}{K-2}}+N^{\frac{-2}{K}}+(N^n\lambda)^{\frac{-2}{K}}\right).
\end{align*}
\end{Le}
\section{Convergence of ergodic weighted averages with natural density}
 Let $f$ be a measurable function and denote by $S_N(f)$ the weighted sums associated with $f$, i.e.  
 \[S_N(f)= \sum_{k=0}^{N-1}w_kT^{u_k}(f).\]
Let $(A_N)_{N>0}$ be a sequence converging to infinity as $N$ goes to infinity. In this section, we prove Theorems 1, 2, 3 and their corollaries. Our proofs can be divided in two types. The first one concerns Theorem \ref{Th1}, Theorem \ref{Th2} and the first statement of Theorem \ref{Th3} and uses M\'oricz's lemma as a crucial role. The second one concerns the second statement of Theorem \ref{Th3} and uses another method.  

The first type of proof operates in the following way.  
Let $(N_k)$ be an increasing sequence of integers. Let $N$ be a sufficiently large integer and let $j \in \NN^*$ be such that $N_j<N\leq N_{j+1}$ and $N_{j+1}<CN_j$.  The main idea is to prove that for any $f \in L^2(\mu)$,
\begin{align}{\label{eq1}}
 \lim\limits_{j\rightarrow\infty} \left\rvert  \frac{S_{N_j}(f)}{A_{N_j}}\right\rvert=0 \quad \mu-a.e.
\end{align}
and that
\begin{align}{\label{eq2}}
 \lim\limits_{j\rightarrow\infty} \sup_{N=N_{j}+1}^{N_{j+1}}\left\rvert \frac{S_{N}(f)}{A_{N}}-\frac{S_{N_j}(f)}{A_{N_j}}\right\rvert=0 \quad \mu-a.e..  
\end{align}
 Proving Equations \eqref{eq1} and \eqref{eq2} is sufficient since for any $j \in \NN$, we have
\begin{align}{\label{osillation}}
    \left\rvert\frac{S_{N}(f)}{A_{N}}\right\rvert&\leq  \left\rvert \frac{S_{N}(f)}{A_{N}}-\frac{S_{N_j}(f)}{A_{N_j}}\right\rvert+\left\rvert\frac{S_{N_j}(f)}{A_{N_j}}\right\rvert \notag \\
    &\leq \sup_{N=N_{j}+1}^{N_{j+1}}\left\rvert\frac{S_{N}(f)}{A_{N}}-\frac{S_{N_j}(f)}{A_{N_j}}\right\rvert+\left\rvert\frac{S_{N_j}(f)}{A_{N_j}}\right\rvert.
\end{align}
To deal with \eqref{eq2}, we introduce for any $k\geq1$ the following random variable
 \[X_k=\frac{S_{k-1}(f)}{A_{k-1}}-\frac{S_{k}(f)}{A_{k}}.\]
 Notice that for any $n,m \in \llbracket N_j,N_{j+1}\rrbracket$ with $m<n$,  
\begin{align*}
    \mathbb{E}\left|\sum_{k=m+1}^{n} X_k\right|^2&= \left\rvert \left\rvert\frac{S_{m}(f)}{A_m}-\frac{S_{n}(f)}{A_n}\right\rvert \right\rvert_{2,\mu}^2
    \leq \left\rvert \left\rvert f \right\rvert \right\rvert_{2,\mu}^2 \sup_{\theta \in \mathbb{R}} L^2(m,n,\theta),
\end{align*}
where the last inequality comes from Lemma 1, with
\begin{align}{\label{eqlmn}}
L(m,n,\theta):= \left\rvert\frac{V_{m}(\theta)}{A_m}-\frac{V_{n}(\theta)}{A_n}\right\rvert.    
\end{align}
The main idea is to bound $\sup_{\theta \in \mathbb{R}} L^2(m,n,\theta)$ by $Cg^\lambda(m,n)$ for $\lambda\geq 1$ and for some function $g$ of the form  \[g(m,n)=\int^{\frac{1}{m}}_{\frac{1}{n}}\frac{1}{x\log^{L} \frac{1}{x}}\mathrm{d}x,\]
with $L>1$. Indeed, if we do it, then according to  M\'oricz's lemma we have
 \[\mathbb{E} \left(\sup_{N=N_{j}+1}^{N_{j+1}}\left\rvert \frac{S_{N}(f)}{A_{N}}-\frac{S_{N_j}(f)}{A_{N_j}} \right\rvert^2\right) \leq C g^\lambda(N_j,N_{j+1}).\]
 By taking the sum over $j$, and by applying Beppo Levi theorem, this proves \eqref{eq2}.
 
 In our computations, we also prove that for any $0<l<L-1$,
\[\sum_{j\geq 1} j^l\sup_{N=N_{j}+1}^{N_{j+1}}\left\rvert \frac{S_{N}(f)}{A_{N}}-\frac{S_{N_j}(f)}{A_{N_j}} \right\rvert^2<\infty \quad \mu-a.e.. \]
\begin{prooft}{Theorem \ref{Th1}}
 Let $N_j=2^j$. Notice that, for any $m,n \in  \llbracket 2^j,2^{j+1}\rrbracket$ with $m<n$, we have $m<n\leq 2m$. First, we deal with \eqref{eq1} for $A_N=N^{\alpha+\delta}\log^H N$. It follows from Lemma 1 and from the assumption of Theorem 1, that
\begin{align*}{\label{cv sous suite}}
   \left\rvert \left\rvert \frac{S_{N_j}(f)}{N_j^{\alpha+\delta}\log^H{N_j}}\right\rvert \right\rvert_{2,\mu}^2 \leq  \int_\mathbb{R} \left\rvert \frac{V_{N_j}(\theta)}{N_j^{\alpha+\delta}\log^H{N_j}}\right\rvert^2 \mu_f(\mathrm{d}\theta)\leq C ||f||_{2,\mu}^2\frac{1}{\log^{2H-2\beta}N_j}.
\end{align*}
 Since $H>\frac{1}{2}+\beta$, this implies
\begin{align*}
    \sum_{j\geq1} \left\rvert \left\rvert \frac{S_{N_j}(f)}{N_j^{\alpha+\delta}\log^HN_j}\right\rvert \right\rvert_{2,\mu}^2<\infty.
\end{align*}
Consequently
\[\lim\limits_{j\rightarrow\infty}\frac{S_{N_j}(f)}{N_j^{\alpha+\delta}\log^HN_j}=0 \quad \mu-a.e..\]\\
Now, we have to prove that the limit of the oscillation as defined in \eqref{eq2} is equal to $0$ $\mu$-a.e..
 For any $\theta \in \mathbb{R}$, we write
\begin{align}{\label{majoration-1}}
 L(m,n,\theta)&=\left\rvert\frac{1}{n^{\alpha+\delta} \log^Hn}\sum_{k=0}^{n-1}w_ke^{2i\pi \theta u_k}-\frac{1}{m^{\alpha+\delta} \log^Hm}\sum_{k=0}^{m-1}w_ke^{2i\pi \theta u_k} \right\rvert \notag \\
 &\leq \left(\frac{1}{m^{\alpha+\delta} \log^Hm}-\frac{1}{n^{\alpha+\delta} \log^Hn}\right) \left|\sum_{k=0}^{m-1}w_ke^{2i\pi \theta u_k}\right|+\frac{1}{n^{\alpha+\delta} \log^Hn}\left|\sum_{k=m}^{n-1}w_ke^{2i\pi \theta u_k}\right|.
\end{align}
 According to Condition (H1), we have
\begin{align}{\label{majoration0}}
 \sup_{ \theta \in \mathbb{R}} L(m,n,\theta) &\leq C   \left(\frac{1}{m^{\alpha+\delta}}-\frac{1}{n^{\alpha+\delta}}\right)\frac{m^{\alpha+\delta}\log^\beta m}{\log^Hm}+\frac{n^\delta(n-m)^{\alpha}\log^\beta n}{n^{\alpha+\delta} \log^H n} \notag \\
   & \leq C\frac{(n-m)^\alpha}{n^{\alpha} \log^{H-\beta} n}.
\end{align}
Now we consider two cases.\\
\textbf{Case $\alpha>\frac{1}{2}$}. Let $\epsilon>0$. For any $n,m \in \llbracket 2^j,2^{j+1}\rrbracket$ with $m<n$, and for any $\theta \in \mathbb{R}$, we have
\begin{align}{\label{majoration1}}
 L^2(m,n,\theta)&=\left(  L(m,n,\theta)  L(m,n,\theta)^{\frac{1-\epsilon}{1+\epsilon}} \right)^{1+\epsilon}\notag \\
  &\leq C\left(\left(\frac{(n-m)^\alpha}{n^\alpha \log^{H-\beta} n}\right)\left( \frac{(n-m)^\alpha}{n^\alpha \log^{H-\beta} n}\right)^{\frac{1-\epsilon}{1+\epsilon}}  \right)^{1+\epsilon} \notag \\
 &= C\left(\left(\frac{1}{m}-\frac{1}{n}\right)\frac{mn(n-m)^{\alpha-1}}{n^\alpha \log^{H-\beta} n}\left( \frac{(n-m)^\alpha}{n^\alpha \log^{H-\beta} n}\right)^{\frac{1-\epsilon}{1+\epsilon}}  \right)^{1+\epsilon} \notag \\
 &=C\left(\int_{\frac{1}{n}}^{\frac{1}{m}} \frac{mn(n-m)^{\alpha-1}}{n^\alpha \log^{H-\beta} n}\left( \frac{(n-m)^\alpha}{n^\alpha \log^{H-\beta} n}\right)^{\frac{1-\epsilon}{1+\epsilon}}\mathrm{d}x \right)^{1+\epsilon} \notag \\
 &\leq C\left(\int_{\frac{1}{n}}^{\frac{1}{m}} \frac{(n-m)^{\alpha\frac{2}{1+\epsilon}-1}}{n^{ \alpha\frac{2}{1+\epsilon}-2}\log^{\frac{2(H-\beta)}{1+\epsilon}}n}\mathrm{d}x \right)^{1+\epsilon}.
\end{align}
Since $\alpha> \frac{1}{2}$, we have $\alpha\frac{2}{1+\epsilon}-1>0$, by taking $\epsilon$ small enough. In particular, we have  $(n-m)^{\alpha\frac{2}{1+\epsilon}-1} \leq n^{\alpha\frac{2}{1+\epsilon}-1}$. Thus
\begin{align}{\label{majoration2}}
 \sup_{ \theta \in \mathbb{R}} L^2(m,n,\theta)\leq  C\left(\int_{\frac{1}{n}}^{\frac{1}{m}} \frac{n}{\log^{\frac{2(H-\beta)}{1+\epsilon}}n}\mathrm{d}x \right)^{1+\epsilon}.
\end{align}
For any $\frac{1}{n} \leq x \leq  \frac{1}{m}$ and $m < n \leq 2m $, we notice that $\frac{1}{n} \leq x \leq  \frac{2}{n}$. Therefore    
\begin{align}{\label{majoration3}}
 \sup_{ \theta \in \mathbb{R}} L^2(m,n,\theta)\leq   C \left( \int_{\frac{1}{n}}^{\frac{1}{m}}\frac{1}{x\log^{\frac{2(H-\beta)}{1+\epsilon}}\frac{1}{x}} \mathrm{d}x\right)^{1+\epsilon}.
\end{align}
Consequently, for any $m,n$ belonging in  $\llbracket 2^j,2^{j+1}\rrbracket$ with $m<n$, we have
\begin{align*}
    \mathbb{E}\left|\sum_{k=m+1}^{n} X_k\right|^2 \leq  C \left|\left|f\right|\right|_{2,\mu}^2 \left( \int_{\frac{1}{n}}^{\frac{1}{m}}\frac{1}{x\log^{\frac{2(H-\beta)}{1+\epsilon}}\frac{1}{x}} \mathrm{d}x\right)^{1+\epsilon}.
\end{align*}
Applying M\'{o}ricz's lemma with $\lambda=1+\epsilon$ and
\[g_\epsilon(m,n)=\int_{\frac{1}{n}}^{\frac{1}{m}}\frac{1}{x\log^{\frac{2(H-\beta)}{1+\epsilon}}\frac{1}{x}} \mathrm{d}x,\]
we obtain
\begin{align*}
 \mathbb{E}\left(\sup_{N=m+1}^{n}\left\rvert\frac{S_{m}(f)}{m^{\alpha+\delta}\log^Hm}-\frac{S_{N}(f)}{N^{\alpha+\delta}\log^HN} \right\rvert^2\right)\leq C(\epsilon)\left|\left|f\right|\right|_{2,\mu}^2 g_\epsilon^{1+\epsilon}(m,n).  
\end{align*}
Since the previous inequality is true for any $m,n \in \llbracket 2^j,2^{j+1}\rrbracket$ with $m<n$, we have
\begin{align*}
 \mathbb{E}\left(\sup_{N=N_j+1}^{N_{j+1}}\left\rvert\frac{S_{N_j}(f)}{N_j^{\alpha+\delta}\log^HN_j}-\frac{S_{N}(f)}{N^{\alpha+\delta}\log^HN} \right\rvert^2\right)\leq C(\epsilon)\left|\left|f\right|\right|_{2,\mu}^2 g_\epsilon^{1+\epsilon}(N_j,N_{j+1}).  
\end{align*}
 Notice that the function $x \mapsto \frac{1}{x\log^{\frac{2(H-\beta)}{1+\epsilon}}\frac{1}{x}}$ is locally integrable in $0$ because $H>\frac{1}{2}+\beta$. Moreover there exists a constant $C>0$ such that $g_\epsilon^{1+\epsilon}(N_j,N_{j+1})\leq C g_\epsilon(N_j,N_{j+1})$. Therefore  \[\sum_{j\geq1} g_\epsilon^{1+\epsilon}(N_j,N_{j+1})<\infty.\]
 By Beppo Levi's theorem, we have \[\lim\limits_{j\rightarrow\infty} \sup_{N=N_{j}+1}^{N_{j+1}}\left\rvert \frac{S_{N}(f)}{N^{\alpha+\delta}\log^HN}-\frac{S_{N_j}(f)}{N_j^{\alpha+\delta}\log^HN_j}\right\rvert=0 \quad \mu-a.e..\]
 
 Now, we prove the strong maximal inequality. To do it, we use \eqref{osillation}, and we take the supremum over $j$. This gives
\begin{align*}
 \sup_{N>1}\left|\frac{1}{N^{\alpha+\delta}\log^HN}\sum_{k=0}^{N-1} w_k T^{u_k}(f)\right|^2 \leq 2\sup_{j\geq1}\sup_{N=N_{j}+1}^{N_{j+1}}\left\rvert\frac{S_{N}(f)}{N^{\alpha+\delta}\log^H{N}}-\frac{S_{N_j}(f)}{N_j^{\alpha+\delta}\log^H{N_j}}\right\rvert^2 \\
 +2\sup_{j\geq1}\left\rvert\frac{S_{N_j}(f)}{N_j^{\alpha+\delta}\log^H{N_j}}\right\rvert^2.  
\end{align*}
 Integrating over $\mu$, we get  
 \begin{align*}
   \left|\left|\sup_{N>1}\left|\frac{1}{N^{\alpha+\delta}\log^HN}\sum_{k=0}^{N-1} w_k T^{u_k}(f)\right| \, \right|\right|_{2,\mu}^2 \leq C \left|\left|f\right|\right|_{2,\mu}^2\sum_{j\geq1} g_\epsilon^{1+\epsilon}(N_j,N_{j+1}) \\
   +C\left|\left|f\right|\right|_{2,\mu}^2 \sum_{j\geq1}\frac{1}{\log^{2H-2\beta}N_j}.
\end{align*}
\textbf{Case $\alpha=\frac{1}{2}$}. According to \eqref{majoration0}, for any $m,n\in \llbracket 2^j,2^{j+1}\rrbracket$ with $m<n$, we have
\[\sup_{ \theta \in \mathbb{R}} L^2(m,n,\theta)\leq C\frac{n-m}{n \log^{2(H-\beta)}n}.\]
Similarly to \eqref{majoration1}, \eqref{majoration2} and \eqref{majoration3} with $\alpha=\frac{1}{2}$ and $\epsilon=0$, we obtain
\begin{align*}
    \mathbb{E}\left|\sum_{k=m+1}^{n} X_k\right|^2 \leq  C \left|\left|f\right|\right|_{2,\mu}^2 \int_{\frac{1}{n}}^{\frac{1}{m}}\frac{1}{x\log^{2(H-\beta)}\frac{1}{x}} \mathrm{d}x.
\end{align*}
 Applying M\'oricz's lemma with $\lambda=1$ and $g(m,n)=\int_{\frac{1}{n}}^{\frac{1}{m}} \frac{1}{x\log^{2(H-\beta)}\frac{1}{x}}\mathrm{d}x$, we obtain for any $m,n\in \llbracket 2^j,2^{j+1}\rrbracket$ with $m<n$, that
\begin{align*}
 \mathbb{E}\left(\sup_{N=m+1}^{n}\left\rvert\frac{S_{m}(f)}{m^{\alpha+\delta}\log^Hm}-\frac{S_{N}(f)}{N^{\alpha+\delta}\log^HN} \right\rvert^2\right)\leq C\left|\left|f\right|\right|_{2,\mu}^2 \log^2(2(n-m)) g(m,n).  
\end{align*}
Since $\frac{n}{2}\leq m$, we have $\log^2(2(n-m))\leq \log^2n$. In the same spirit as \eqref{majoration2} and \eqref{majoration3}, we can write
\[\log^2(2(n-m)) g(m,n) \leq \int_{\frac{1}{n}}^{\frac{1}{m}}\frac{1}{x\log^{2(H-\beta)-2}\frac{1}{x}} \mathrm{d}x.\]
 Since the function  $x \mapsto  \frac{1}{x\log^{2(H-\beta)-2}\frac{1}{x}}$ is locally integrable in $0$ (because $H>\frac{3}{2}+\beta$), it follows from  Beppo Levi's theorem that \[\lim\limits_{j\rightarrow\infty} \sup_{N=N_{j}+1}^{N_{j+1}}\left\rvert \frac{S_{N}(f)}{N^{\alpha+\delta}\log^HN}-\frac{S_{N_j}(f)}{N_j^{\alpha+\delta}\log^HN_j}\right\rvert=0 \quad \mu-a.e..\]

Moreover
\begin{align*}
   \left|\left|\sup_{N>1}\left|\frac{1}{N^{\alpha+\delta}\log^HN}\sum_{k=0}^{N-1} w_k T^{u_k}(f)\right| \, \right|\right|_{2,\mu}^2 \leq C\left|\left|f\right|\right|_{2,\mu}^2 \sum_{j\geq1} \int_{\frac{1}{N_{j+1}}}^{\frac{1}{N_j}} \frac{1}{x\log^{2(H-\beta)-2}\frac{1}{x}}\mathrm{d}x \\
   +C\left|\left|f\right|\right|_{2,\mu}^2 \sum_{j\geq1}\frac{1}{\log^{2H-2\beta}N_j} .
\end{align*}
This concludes the proof of Theorem 1.
\end{prooft}

\begin{prooft}{Corollary \ref{cor1}}
By Abel's summation formula, we have
\begin{align*}
  \sum_{k=2}^{N-1}\frac{w_k}{k^{\alpha+\delta}\log^Hk} T^{u_k}(f)= \sum_{k=2}^{N-2}\left(\frac{1}{k^{\alpha+\delta}\log^Hk}-\frac{1}{(k+1)^{\alpha+\delta}\log^H(k+1)}\right)\sum_{j=2}^{k} w_j T^{u_j}(f)\\
  +\frac{1}{N^{\alpha+\delta}\log^HN}\sum_{k=2}^{N-1} w_k T^{u_k}(f).
\end{align*}
 According to Theorem \ref{Th1}, the last term of the right-hand side converges to $0$ $\mu$-a.e.. Moreover, according to Lemma 1 and the fact that
 \[\left\lvert\frac{1}{k^{\alpha+\delta}\log^Hk}-\frac{1}{(k+1)^{\alpha+\delta}\log^H(k+1)}\right\rvert\leq C \frac{1}{k^{\alpha+\delta+1}\log^{H}k},\] it follows from the assumption of Corollary 1, that
 \begin{multline*}
   \left|\left|\sum_{k>1}\left(\frac{1}{k^{\alpha+\delta}\log^Hk}-\frac{1}{(k+1)^{\alpha+\delta}\log^H(k+1)}\right)\left\lvert\sum_{j=2}^{k} w_j T^{u_j}(f)\right\rvert \,\right|\right|_{2,\mu} \\
   \begin{split}
   &\leq C\sum_{k>1} \frac{1}{k^{\alpha+\delta+1}\log^{H}k} \left|\left|\sum_{j=1}^{k} w_j T^{u_j}(f)\right|\right|_{2,\mu} \\
   &\leq C ||f||_{2,\mu} \sum_{k>1} \frac{1}{k\log^{H-\beta}k}.
   \end{split}
 \end{multline*}
 The last term is finite because $H>1+\beta$. According to Beppo Levi's theorem, this proves that $\sum_{k>1} \frac{w_k}{k^{\alpha}\log^Hk} T^{u_k}(f)$ exists $\mu$-a.e.. Moreover,
 \begin{multline*}
    \left\lvert\left\lvert \sup_{N>1} \left\lvert\sum_{k=2}^{N} \frac{w_k}{k^{\alpha}\log^H k} T^{u_k}(f) \right\rvert \,\right\rvert\right\rvert_{2,\mu}
    \leq    \left\lvert\left\lvert \sum_{k>1}\left(\frac{1}{k^{\alpha}\log^Hk}-\frac{1}{(k+1)^{\alpha}\log^H(k+1)}\right)\left|\sum_{j=2}^{k} w_j T^{u_j}(f)\right|   \,\right\rvert\right\rvert_{2,\mu} \\
    +  \left\lvert\left\lvert \sup_{N>1}\left|\frac{1}{N^{\alpha}\log^HN}\sum_{k=2}^{N-1} w_k T^{u_k}(f)\right| \, \right\rvert\right\rvert_{2,\mu}.
    \end{multline*}
    Therefore \[ \left\lvert\left\lvert \sup_{N>1} \left\lvert\sum_{k=2}^{N} \frac{w_k}{k^{\alpha}\log^H k} T^{u_k}(f) \right\rvert \,\right\rvert\right\rvert_{2,\mu} \leq C||f||_{2,\mu} .\]

\end{prooft}
\begin{prooft}{Theorem \ref{Th2}} Here $A_N=N^{\frac{\alpha+1}{2}}\log^H N$. Choosing $N_j=2^j$ and
following the same lines as in the proof of Theorem \ref{Th1}, we obtain
\[\left\rvert \left\rvert \frac{S_{N_j}(f)}{N_j^{\frac{\alpha+1}{2}}\log^H{N_j}}\right\rvert \right\rvert_{2,\mu}^2 \leq C ||f||_{2,\mu}^2\frac{1}{2^{j(1-\alpha)}j^{2(H-\beta)}}.\]
Recall that $L(m,n,\theta)$ is defined in \eqref{eqlmn}. To bound $\sup_{\theta \in \mathbb{R}}L^2(m,n,\theta)$, we first notice that
\begin{align*}
    \sup_{\theta \in \mathbb{R}} L(m,n,\theta) &\leq \left(\frac{1}{m^{\frac{\alpha+1}{2}}\log^Hm}-\frac{1}{n^{\frac{\alpha+1}{2}}\log^Hn}\right)m^\alpha \log^\beta m+\frac{|V_{m,n}(\theta)|}{n^{\frac{\alpha+1}{2}}\log^Hn} \\
    &\leq \left(\frac{1}{m^{\frac{\alpha+1}{2}}\log^Hm}-\frac{1}{n^{\frac{\alpha+1}{2}}\log^Hn}\right)m^\alpha \log^\beta m+\frac{n-m}{n^{\frac{\alpha+1}{2}}\log^Hn} \\
& \leq C \frac{n-m}{n^{\frac{\alpha+1}{2}}\log^{H-\beta}n}\\
&\leq C\int_{\frac{1}{n}}^{\frac{1}{m}} \frac{n^2}{n^{\frac{\alpha+1}{2}}\log^{H-\beta}n}\mathrm{d}x.
\end{align*}
Moreover
\begin{align*}
    \sup_{\theta \in \mathbb{R}} L(m,n,\theta) \leq C \frac{n^\alpha \log^\beta n}{n^{\frac{\alpha+1}{2}}\log^Hn}.
\end{align*}
Therefore
\begin{align*}
 \sup_{ \theta \in \mathbb{R}} L^2(m,n,\theta)&\leq C \int_{\frac{1}{n}}^{\frac{1}{m}} \frac{n^2}{n^{\frac{\alpha+1}{2}}\log^{H-\beta}n}\times \frac{n^\alpha}{n^{\frac{\alpha+1}{2}}\log^{H-\beta}n}\mathrm{d}x \\
 & =C \int_{\frac{1}{n}}^{\frac{1}{m}} \frac{n}{\log^{2(H-\beta)}n}\mathrm{d}x.
\end{align*}
Consequently, for any $m,n$ in $\llbracket 2^j,2^{j+1}\rrbracket$ with $m<n$
\begin{align*}
    \mathbb{E}\left|\sum_{k=m+1}^{n} X_k\right|^2 \leq  C \int_{\frac{1}{n}}^{\frac{1}{m}} \frac{1}{x\log^{2(H-\beta)}\frac{1}{x}}\mathrm{d}x.
\end{align*}
By M\'{o}ricz's lemma applied to $\lambda=1$ and
$g(m,n)=\displaystyle{\int_{\frac{1}{n}}^{\frac{1}{m}} \frac{1}{x\log^{2(H-\beta)}\frac{1}{x}}\mathrm{d}x}$,
we obtain
\[\mathbb{E}\left(\sup_{N=m+1}^{n}\left\rvert\frac{S_{m}(f)}{m^{\frac{\alpha+1}{2}}\log^Hm}-\frac{S_{N}(f)}{N^{\frac{\alpha+1}{2}}\log^HN} \right\rvert^2\right)\leq  C  \int_{\frac{1}{n}}^{\frac{1}{m}}\frac{1}{x\log^{2(H-\beta)-2}\frac{1}{x}} \mathrm{d}x, \]
which converges if $H>\frac{3}{2}+\beta.$ The end of the proof follows the same lines as the one of Theorem \ref{Th1}.
\end{prooft}
\begin{prooft}{Corollary \ref{cor2}} We proceed as in the proof of   Corollary \ref{cor1}, by using Abel's summation formula and Theorem \ref{Th2}.
\end{prooft}
\begin{prooft}{Theorem \ref{Th3}}
First assume that $\beta>1$. We only give a sketch of proof. Taking $N_j=2^j$, we can easily prove that
 \[\left\rvert \left\rvert \frac{S_{N_j}(f)}{N_j}\right\rvert \right\rvert_{2,\mu}^2 \leq C \frac{1}{\log^{2\beta}N_j},\]
 and that
 \[\mathbb{E}\left(\sup_{N=N_j+1}^{N_{j+1}}\left\rvert\frac{S_{N_{j}}(f)}{N_{j}}-\frac{S_{N}(f)}{N} \right\rvert^2\right)\leq C(\epsilon)\left|\left|f\right|\right|_{2,\mu}^2\left( \int_{\frac{1}{n}}^{\frac{1}{m}}\frac{1}{x\log^{\beta\frac{1-\epsilon}{1+\epsilon}}\frac{1}{x}} \mathrm{d}x\right)^{1+\epsilon}.
 \]
 The end of the proof also follows the same lines as the one of Theorem \ref{Th1}.
 
Now assume that $\frac{1}{2}<\beta\leq 1$ and let $\rho>1$. For $N$ large enough there exists an integer $k$ such that
$\rho^{k^{1-\epsilon} \log k}<N<\rho^{(k+1)^{1-\epsilon}\log(k+1)}$, where $0<\epsilon \leq \frac{1}{2}$ is chosen in such a way that $\beta=\frac{1}{2(1-\epsilon)}$. Letting $\rho(k)=\rho^{k^{1-\epsilon} \log k}$, we have
\begin{align*}
    \left\rvert \frac{1}{N}\sum_{n=0}^{N-1} w_n T^{u_n}(f) \right\rvert &\leq
    \left\rvert \frac{1}{N}\sum_{n=0}^{\lfloor \rho(k)\rfloor-1} w_n T^{u_n}(f) \right\rvert +  \left\rvert \frac{1}{N}\sum_{n=\lfloor \rho(k)\rfloor}^{\lfloor\rho(k+1)\rfloor-1}  w_n T^{u_n}(f) \right\rvert
 \notag \\
    &\leq \left\rvert \frac{1}{\rho(k)}\sum_{n=0}^{\lfloor \rho(k)\rfloor-1} w_n T^{u_n}(f) \right\rvert +  \frac{1}{\rho(k)}\sum_{n=\lfloor \rho(k)\rfloor}^{\lfloor\rho(k+1)\rfloor-1}  \left\rvert T^{u_n}(f) \right\rvert,
\end{align*}
where $\lfloor \cdot \rfloor$ denote the integer part.
We prove below that the two terms of the right-hand side converge to $0$ $\mu$-a.e. as $k$ goes to infinity. To deal with the first one, we use the spectral lemma and our assumption. For any $f \in L^2(\mu)$, this gives
\[\left\rvert\left\rvert \frac{1}{\rho(k)}\sum_{n=0}^{\lfloor \rho(k)\rfloor-1} w_n T^{u_n}(f)\right\rvert \right\rvert_{2,\mu}^2 \leq \frac{ C\left\rvert\left\rvert f\right\rvert \right\rvert_{2,\mu}^2}{\log^{2\beta} \rho} \times \frac{1}{ k^{(1-\epsilon)2\beta}\log^{2\beta}k}.\]
Since $\beta=\frac{1}{2(1-\epsilon)}$, the right-hand side is the term of a convergent series. This shows that  
\[\lim\limits_{k\rightarrow\infty}\frac{1}{\rho(k)}\sum_{n=0}^{\lfloor \rho(k)\rfloor-1} w_n T^{u_n}(f)=0 \quad \mu-a.e..\]

To deal with the second term,
consider a sequence $(a_k)$ converging to infinity (this sequence will be defined later) such that $f=f_1+f_2$, where $f_1=f \ind{|f|\leq a_k}$ and  $f_2=f \ind{|f|> a_k}$. In particular, we have
\[T^{u_n}(f)=T^{u_n}(f_1)+T^{u_n}(f_2).\]
Therefore
\begin{align*}
\frac{1}{\rho(k)}\sum_{n=\lfloor \rho(k)\rfloor}^{\lfloor\rho(1+k)\rfloor-1} \left\rvert T^{u_n}(f_1) \right\rvert&\leq \frac{1}{\rho(k)}\sum_{n=\lfloor \rho(k)\rfloor}^{\lfloor\rho(1+k)\rfloor-1}\left\rvert T^{u_n}(f) \right\rvert \ind{|T^{u_n}(f)|\leq a_k} \notag \\
 &\leq a_k \frac{\lfloor\rho(1+k)\rfloor-\lfloor{\rho(k)}\rfloor}{\rho(k)} \notag \\
 &\leq Ca_k \frac{\log k}{k^\epsilon},
\end{align*}
where the last inequality comes from the mean value theorem. Now, choose $a_k=\frac{k^\epsilon}{\log^\delta k}$ ($\delta>1$). The above computations show that
$ \frac{1}{\rho(k)}\sum_{n=\lfloor \rho(k+1)\rfloor}^{\lfloor\rho(1+k)\rfloor-1} \left\rvert T^{u_n}(f_1) \right\rvert$ converges to $0$ as $k$ goes to infinity, $\mu$-a.e.. Moreover  
\begin{align*}
 \frac{1}{{\rho(k)}}\sum_{n=\lfloor {\rho(k)}\rfloor}^{\lfloor{\rho(1+k)}\rfloor-1} \left\rvert T^{u_n}(f_2) \right\rvert &\leq \frac{1}{{\rho(k)}}\sum_{n=\lfloor {\rho(k)}\rfloor}^{\lfloor{\rho(1+k)}\rfloor-1}\left\rvert T^{u_n}(f) \right\rvert \ind{|f|> a_k}.
\end{align*}
By the Cauchy?Schwarz inequality and because the square function is convex, we obtain    
\begin{align*}
\sum_{k\geq 1}\int_{X}  \frac{1}{{\rho(k)}}\sum_{n=\lfloor {\rho(k)}\rfloor}^{\lfloor{\rho(1+k)}\rfloor-1} \left\rvert T^{u_n}(f_2) \right\rvert \mathrm{d}\mu &\leq C\sqrt{\sum_{k\geq 1}\int_{X}\frac{|f|^2}{k\log^{\delta}k}\mathrm{d}\mu} \times \sqrt{\sum_{k\geq 1}\frac{k\log^{2+\delta}k}{k^{2\epsilon}} \mu(|f|> a_k)}.
\end{align*}
Since $\sum_{k\geq 1}\frac{1}{k\log^{\delta}k}$ is a convergent series, we have
\[ \sqrt{\sum_{k\geq 1}\int_{X}C\frac{|f|^2}{k\log^{1+\epsilon}k}\mathrm{d}\mu} \leq C\left\rvert\left\rvert f\right\rvert \right\rvert_{2,\mu}.\]
To conclude the proof, it is enough to show that $\sum_{k\geq 1}\frac{k\log^{2+\delta}k}{k^{2\epsilon}} \mu(|f|> a_k)<\infty$. To do it, we write for some $p>0$,
\begin{align*}
  \mu\left(|f|>\frac{k^\epsilon}{\log^\delta {k}}\right)\leq \mu\left(|f|^p\log^\tau(|f|+1)>C_{\epsilon,\tau} \frac{k^{p\epsilon}}{\log^{p\delta} k} \log^\tau k  \right),
\end{align*}
 where the last inequality comes from the fact that
\begin{align*}
    |f|>\frac{k^\epsilon}{\log^\delta {k}} &\iff \log^\tau(|f|+1)> \left(\log k^\epsilon -\log \log^\delta k +\log\left(1+\frac{\log^\delta k}{k^\epsilon}\right) \right)^\tau.
\end{align*}
This gives that
\begin{align*}
    \sum_{k\geq 1}\frac{k\log^{2+\delta}k}{k^{2\epsilon}} \mu(|f|> a_k) \leq C \int_X \lvert f\rvert^{p}\log^\tau(1+\lvert f\rvert)\mathrm{d}\mu \sum_{k\geq 1}\frac{1}{k^{(2\epsilon+p\epsilon-1)}\log^{(\tau-2-\delta-p\delta)}k}.
\end{align*}
 The last series converges when $p\epsilon+2\epsilon-1=1$ and $\tau-2-\delta-p\delta>1$, i.e. when $p=\frac{2}{2\beta-1}$ and $\tau>4+\frac{2}{2\beta-1}$. Taking $f \in L^2(\mu)$ so that $\int_X \lvert f\rvert^{p}\log^\tau(1+\lvert f\rvert)\mathrm{d}\mu<\infty$, we deduce that $ \frac{1}{\rho(k)}\sum_{n=\lfloor \rho(k)\rfloor}^{\lfloor{\rho(1+k)}\rfloor-1} \left\rvert T^{u_n}(f_2) \right\rvert$ converges to $0$ as $k$ goes to infinity, $\mu$-a.e..
This concludes the proof of Theorem \ref{Th3}.
\end{prooft}
\begin{prooft}{Corollary \ref{cor3}} We proceed as in the proof of   Corollary \ref{cor1}, by using Abel's summation formula and Theorem \ref{Th3}.
\end{prooft}
\section{Convergence of ergodic weighted averages with respect to the harmonic density}
In this section, we will investigate the almost everywhere convergence of some averages with respect to the harmonic density. More precisely,
let $(w_k)$ be a sequence of complex numbers and let $(u_k)$ be a sequence of integers. For any integers $M,N$ such that $M\leq N$, denote by
   \[V_{M,N}^*(\theta)= \sum_{k=M}^{N}\frac{w_k}{k}e^{2i\pi \theta u_k}.\]
   If $M=1$, we will simply write $V_{M,N}^*(\theta)=V_{N}^*(\theta)$. We also denote by
\[S_N^*(f)= \sum_{k=1}^{N}\frac{w_k}{k}T^{u_k}(f),\]
where $f$ is a measurable function. The following result is related to Theorem 1, and is stated in the context of the harmonic density.
\begin{Th}{\label{Th4}} Let $(w_k)$ be a sequence of complex numbers and let $(u_k)$ be a sequence of integers.
 Suppose that there exists $ \frac{1}{2}\leq \alpha<1$ such that for any $M\leq N$
\begin{align}{\label{ass:main1}}
 \sup_{\theta \in \mathbb{R}}\left|V_{M,N}^*(\theta)\right|\leq C\left(\log N-\log M\right)^\alpha.  
\end{align}
Then for any contraction $T$ on $L^2(\mu)$, any $f \in L^2(\mu)$ and for $H>\frac{1}{2}$, we have
\[
  \lim\limits_{N\rightarrow\infty}\frac{1}{\log^\alpha N \log^H\log N}\sum_{k=1}^{N} \frac{w_k}{k} T^{u_k}(f)=0\quad \mu-a.e.,
  \]
  and
  \[
  \left\lvert\left\lvert \sup_{N>2} \left\lvert\frac{1}{\log^\alpha N \log^H\log N}\sum_{k=1}^{N} \frac{w_k}{k} T^{u_k}(f) \right\rvert \,\right\rvert\right\rvert_{2,\mu} \leq C ||f||_{2,\mu}.
  \]
\end{Th}
\begin{prooft}{Theorem \ref{Th4}}
We proceed as in the proof of Theorem \ref{Th1}. First, we choose $N_j=2^{2^j}$,
then by using  the spectral lemma and assumption (\ref{ass:main1}), we obtain for $H>\frac{1}{2}$  
\begin{align*}
    \sum_{j>1} \left\rvert \left\rvert \frac{S^*_{N_j}(f)}{\log^\alpha N_j\log^H\log N_j}\right\rvert \right\rvert_{2,\mu}^2  \leq C||f||_{2,\mu}^2 \sum_{j>1} \frac{1}{\log^{2H}\log N_j}\leq C||f||_{2,\mu}^2 \sum_{j>1}\frac{1}{j^{2H}}<\infty.
\end{align*}
This implies that
\begin{align}{\label{eq1:main1}}
 \lim\limits_{j\rightarrow\infty}\frac{S^*_{N_j}(f)}{\log^\alpha N_j \log^H\log N_j}=0 \quad \mu-a.e..  
\end{align}
Moreover, for any $m,n$ in $\llbracket N_j,N_{j+1} \rrbracket$, with $m<n$, we have
\begin{align*}
\left\lvert \frac{V^*_{n}(\theta)}{\log^\alpha n \log^H\log n}-  \frac{V^*_{m}(\theta)}{\log^\alpha m \log^H \log m}\right\rvert  
 \leq C \frac{(\log n-\log m)^\alpha}{\log^\alpha n \log^H\log n}.
\end{align*}
Now, let  \[X_k=\frac{S^*_k(f)}{\log^\alpha k \log^H \log k}-\frac{S^*_{k+1}(f)}{\log^\alpha (k+1) \log^H\log (k+1)}.\]
Following the same lines as in the proof of Theorem \ref{Th1}, we obtain for all $\epsilon>0$
\[\mathbb{E}\left\lvert\sum_{k=N_j}^{N_{j+1}-1}X_k\right\rvert^2\leq C \left( \int^{\frac{1}{N_{j}}}_{\frac{1}{ N_{j+1}}}\frac{1}{x \log\frac{1}{x}  \log^{H\frac{1-\epsilon}{1+\epsilon}+H}\log\frac{1}{x}} \mathrm{d}x\right)^{1+\epsilon}.\]
This, together with M\'oricz's lemma with $\lambda=1+\epsilon$, implies
\begin{align*}
 \mathbb{E}\left(\sup_{N=N_{j}+1}^{N_{j+1}}\left\lvert\frac{S^*_{N}(f)}{\log^\alpha N \log^H\log N}-\frac{S^*_{N_{j}}(f)}{\log^\alpha N_{j} \log^H\log N_{j}}\right\rvert^2 \right)\\
 \leq C\left( \int^{\frac{1}{N_{j}}}_{\frac{1}{ N_{j+1}}}\frac{1}{x \log\frac{1}{x}  \log^{H\frac{1-\epsilon}{1+\epsilon}+H}\log\frac{1}{x}} \mathrm{d}x\right)^{1+\epsilon}.  
\end{align*}
Using the Beppo Levi's theorem with $H>\frac{1}{2}$, we deduce that
\begin{align}{\label{eq2:main1}}
\lim\limits_{j\rightarrow\infty}\sup_{N=N_{j}+1}^{N_{j+1}}\left|\frac{S^*_{N}(f)}{\log^\alpha N \log^H\log N}-\frac{S^*_{N_{j}}(f)}{\log^\alpha N_{j+1} \log^H\log N_{j}}\right|=0 \quad \mu-a.e..  
\end{align}
 Theorem 4 follows from \eqref{eq1:main1}, \eqref{eq2:main1} and from the following inequality
\begin{align*}
    \left\rvert\frac{S^*_{N}(f)}{\log^\alpha N \log^H\log N}\right\rvert & \leq \left\lvert\frac{S^*_{N_{j}}(f)}{\log^\alpha N_{j}\log^H\log N_{j}}\right\rvert \\ &+\sup_{N=N_{j}+1}^{N_{j+1}}\left\lvert\frac{S^*_{N}(f)}{\log^\alpha N \log^H\log N}-\frac{S^*_{N_{j}}(f)}{\log^\alpha N_{j}\log^H\log N_{j}}\right\rvert.
\end{align*}
\end{prooft}
By adapting the proof of Theorem 4 and assuming that \[\sup_{\theta \in \mathbb{R}}\left|V_{M,N}^*(\theta)\right| \leq C \left(\underset{p-iterates}{\underbrace{\log \log \dots \log N}}-\underset{p-iterates}{\underbrace{\log \log \dots \log M}}\right)^\alpha,\]
we obtain the same result as Theorem 4 with normalization
$\underset{p-iterates}{\underbrace{\log^\alpha (\log \dots \log N)}} \times \underset{(p+1)-iterates}{\underbrace{\log^H( \log \dots \log N)}}$ instead of $\log^\alpha N \log^H\log N$.
\begin{Th}{\label{Th5}}  Let $(w_k)$ be a bounded sequence of complex numbers and let $(u_n)\subset{\NN}$ be a sequence of integers.
 Suppose that there exists $0\leq\alpha<1$ such that
\begin{align*}
 \sup_{\theta \in \mathbb{R}}\left|V_{N}^*(\theta)\right|\leq C\log^{\alpha} N.  
\end{align*}
 Then for any contraction $T$ on $L^2(\mu)$, any $f \in L^2(\mu)$ and for any $H>\frac{1+\alpha}{2}$, we have
\[
  \lim\limits_{N\rightarrow\infty}\frac{1}{\log^{H}N} \sum_{k=1}^{N} \frac{w_k}{k} T^{u_k}(f)=0\quad \mu-a.e.,
  \]
  and
  \[
  \left\lvert\left\lvert \sup_{N>1} \left\lvert\frac{1}{\log^{H}N}\sum_{k=1}^{N} \frac{w_k}{k} T^{u_k}(f) \right\rvert \,\right\rvert\right\rvert_{2,\mu} \leq C ||f||_{2,\mu}.
  \]
\end{Th}
The proof of the above theorem is omitted and relies on a simple adaptation of the proof of Theorem 1.
\begin{Th}{\label{Prop2}} Let $(w_k)$ be a bounded sequence of complex numbers and let $(u_n)\subset{\NN}$ be a sequence of integers.
 Suppose that, for some $\beta>0$,
\begin{align*}
 \sup_{\theta \in \mathbb{R}}|V^*_N(\theta)|\leq C\frac{\log N}{\log^\beta\log N}.
\end{align*}
\begin{enumerate}
    \item If $\beta>1$, then for any contraction $T$ on $L^2(\mu)$ and any $f \in L^2(\mu)$, we have
\[
\lim\limits_{N\rightarrow\infty}\frac{1}{\log N}\sum_{k=1}^{N} \frac{w_k}{k} T^{u_k}(f)=0 \quad \mu-a.e. \quad \text{and} \quad \left\lvert\left\lvert \sup_{N>1} \left\lvert \frac{1}{\log N}\sum_{k=1}^{N} \frac{w_k}{k} T^{u_k}(f) \right\rvert \,\right\rvert\right\rvert_{2,\mu} \leq C ||f||_{2,\mu}.
\]
\item If $\frac{1}{2}<\beta\leq1$, then for any contraction $T$ on $L^2(\mu)$, any $f \in L^2(\mu)$ such that there exists $\tau>4+\frac{2}{2\beta-1}$ with $\displaystyle{\int_X}\lvert f\rvert^{\frac{2}{2\beta-1}}\log^\tau(1+\lvert f\rvert)\mathrm{d}\mu<\infty$, we have
\begin{align*}
\lim\limits_{N\rightarrow\infty}\frac{1}{\log N}\sum_{k=1}^{N} \frac{w_k}{k} T^{u_k}(f)=0 \quad \mu-a.e..
\end{align*}
\end{enumerate}

\end{Th}
\begin{proof}
We follow the same lines as in the proof of Theorem \ref{Th3}, with \[\rho^{\rho^{k^{1-\epsilon}\log k}}<N\leq \rho^{\rho^{(k+1)^{1-\epsilon}\log(k+1)}}.\]
\end{proof}
\begin{Rk}{\label{Cor4}}
\begin{enumerate}
\item Under the assumptions of Theorem \ref{Th4}, for any contraction $T$ on $L^2(\mu)$, any $f \in L^2(\mu)$ and for $H>1$, we have
\[ \sum_{k>2}\frac{w_k}{k\log^\alpha k \log^H\log k}T^{u_k}(f) \quad \text{exists} \quad \mu-a.e.,\]
and
\[\left\lvert\left\lvert \sup_{N>2} \left\lvert \sum_{k=3}^N\frac{w_k}{k\log^\alpha k \log^H\log k}T^{u_k}(f) \right\rvert \,\right\rvert\right\rvert_{2,\mu} \leq C ||f||_{2,\mu}. \]
\item Under the assumptions of Theorem \ref{Th5}, for any contraction $T$ on $L^2(\mu)$, any $ f \in L^2(\mu)$ and for $H>\frac{1+\alpha}{2}$, we have \[ \sum_{k>1}\frac{w_k}{k\log^{H} k}T^{u_k}(f)\quad \text{exists} \quad \mu-a.e.,\]
and
\[ \left\lvert\left\lvert \sup_{N>1 } \left\lvert \sum_{k=2}^N\frac{w_k}{k\log^{H} k}T^{u_k}(f) \right\rvert \,\right\rvert\right\rvert_{2,\mu} \leq C ||f||_{2,\mu}. \]
\item Under the assumptions of Theorem \ref{Prop2} when $\beta>1$, for any contraction $T$ on $L^2(\mu)$ and any $f \in L^2(\mu)$, we have \[ \sum_{k>1}\frac{w_k}{k\log k}T^{u_k}(f)\quad \text{exists} \quad \mu-a.e.,\]
and \[\left\lvert\left\lvert \sup_{N>1} \left\lvert \sum_{k=2}^N\frac{w_k}{k\log k}T^{u_k}(f) \right\rvert \,\right\rvert\right\rvert_{2,\mu} \leq C ||f||_{2,\mu}. \]
\end{enumerate}
\end{Rk}

\section{Examples}
In this section, we present several examples illustrating our main results.


\begin{Ex}
Let $f \in L^2(\mu)$, $d \in \mathbb{N^*}$ and $\delta>d$. In what follows, we let  $||\delta||=\min(\{\delta\},1-\{\delta\})$, where $\{\delta\}$ denotes the fractional part of $\delta$. As a first example, under suitable conditions over $K$ and $H$ which will be given later,  we obtain that \begin{align}{\label{ex1}}
 \sum_{k\geq1}\frac{e^{2i\pi k^\delta }}{k^{1-||\delta||\frac{1}{3(K-2)}}\log^H k} T^{k^d}(f)  
\end{align}
exists $\mu-a.e.$. Before proving this, we notice that Krause, Lacey and Wierdl \cite{Krause} in their study of the convergence of oscillatory ergodic Hilbert transforms showed that, for any $f \in L^r(\mu), \, 1\leq r < \infty$, the limit  \[\lim\limits_{N\rightarrow\infty}\sum_{k=1}^{N}\frac{e^{2i\pi h(k)}}{k} T^{k}(f)\] exists  $\mu$-a.e., where $h(k)$ is a Hardy field function. Recall that (see Section $3$ in \cite{Krause}) a function of the form $k^\delta=e^{\delta\log k}$ is a Hardy field function. In that case, our result is more precise than \cite{Krause}.

Now, we prove that \eqref{ex1} exists $\mu-a.e.$. According to Corollary 2 it is sufficient to prove that  
\begin{align}{\label{ahm}}
\sup_{\theta \in \mathbb{R}} \left\lvert\sum_{k=1}^{N}e^{2i\pi (k^\delta + \theta k^d)}\right\rvert \leq CN^{1-||\delta||\frac{2}{3(K-2)}},
\end{align}
 where $K=2^n$ with $n-1<\delta<n$.

In what follows, we consider the function $g(x)=\theta x^d+x^\delta$. For any $x \in [N^{1-\epsilon},N]$, where $\epsilon$ is a positive constant which will be defined later, we have
\[g^{(n)}(x)= \delta(\delta-1)\dots (\delta-(n-1))x^{\delta-n}.\]
Consequently,
\[\frac{C(n,\delta)}{N^{n-\delta}}\leq f^{(n)}(x) \leq \frac{C(n,\delta)}{N^{(1-\epsilon)(n-\delta)}},\]
where $C(n,\delta)=\delta(\delta-1)\dots (\delta-(n-1))$. According to Lemma 4 with $\lambda=\frac{C(n,\delta)}{N^{n-\delta}}$ and $h=N^{\epsilon(n-\delta)}$, we get
\[\left\lvert\sum_{k=\lfloor N^{1-\epsilon} \rfloor}^{N}e^{2i\pi (k^\delta + \theta k^d)}\right\rvert  \leq CN\left( \frac{1}{N^{\frac{n-\delta}{K-2}-\epsilon(n-\delta)}} +\frac{1}{N^{\frac{2}{K}-\epsilon(n-\delta)}} +\frac{1}{N^{\frac{2\delta}{K}-\epsilon(n-\delta)}} \right).\]
Now, choosing $\epsilon=\frac{1}{3(K-2)}$ and using the fact that $\delta>1$, we have
\begin{align*}
    \left\lvert\sum_{k=\lfloor N^{1-\epsilon} \rfloor}^{N}e^{2i\pi (k^\delta + \theta k^d)}\right\rvert  
    &\leq CN\left( \frac{1}{N^{||\delta||\left(\frac{1}{K-2}-\epsilon\right)}} +\frac{1}{N^{\frac{2}{K}-\epsilon(n-\delta)}} \right) \\
    & \leq CN\left(\frac{1}{N^{||\delta||\frac{2}{3(K-2)}}} +\frac{1}{N^{\frac{6-n+\delta}{3(K-2)}}} \right).
\end{align*}
Since $||\delta||\leq \frac{1}{2}$ and $\delta>n-1$, we obtain
\begin{align*}
    \left\lvert\sum_{k=\lfloor N^{1-\epsilon} \rfloor}^{N}e^{2i\pi (k^\delta + \theta k^d)}\right\rvert \leq CN\left( \frac{1}{N^{||\delta||\frac{2}{3(K-2)}}} \right).
    \end{align*}
It follows that
\begin{align*}
    \left\lvert\sum_{k=1}^{N}e^{2i\pi (k^\delta + \theta k^d)}\right\rvert &\leq
    \left\lvert\sum_{k=1}^{\lfloor N^{1-\epsilon} \rfloor-1}e^{2i\pi (k^\delta + \theta k^d)}\right\rvert+ \left\lvert\sum_{k=\lfloor N^{1-\epsilon} \rfloor}^{N}e^{2i\pi (k^\delta + \theta k^d)}\right\rvert \\
    & \leq C\left(N^{1-\frac{1}{3(K-2)}}+N^{1-||\delta||\frac{2}{3(K-2)}} \right) \\
    & \leq C N^{1-||\delta||\frac{2}{3(K-2)}}.
\end{align*}
This proves \eqref{ahm}. Notice that \eqref{ahm} and Theorem \ref{Th2} also imply that
\[ \lim\limits_{N\rightarrow\infty}\frac{1}{N^{1-||\delta||\frac{1}{3(K-2)}}\log^H N}\sum_{k=1}^{N} e^{2i\pi k^\delta} T^{k^d}(f)=0 \quad \mu-a.e.,\]
for any contraction $T$ on $L^2(\mu)$, any $f \in L^2(\mu)$ and for $H>\frac{3}{2}$.
\end{Ex}
\begin{Ex}{\label{Ex1}}
In this example, we take $u_k=k$ and $w_k=e^{2i\pi k^\delta}$ for some $1>\delta>0$. First we prove the following inequality  
\begin{align}{\label{Exk}}
\sup_{\theta \in \mathbb{R}} \left\lvert\sum_{k=1}^{N}e^{2i\pi (k^\delta + \theta k)}\right\rvert \leq CN^{1-\frac{\delta}{2}}.
\end{align}
To do it, the main idea is to apply the van der Corput theorem. Let $a=\sqrt{N}$, $b=N$
and $f(x)=\theta x+x^\delta$. In particular, we have
\[f'(x)= \theta+\delta x^{\delta-1} \quad \text{and} \quad f''(x)=\delta(\delta-1) x^{\delta-2}.\]
 Moreover for any $x\in [\sqrt{N},N]$, we have
\[-f''(x)\geq CN^{\delta-2}:=\rho.\]
According to Lemma 3, this gives
\begin{align*}
 \left\lvert\sum_{k=1}^{N}e^{2i\pi (k^\delta + \theta k)}\right\rvert &\leq \left\lvert\sum_{k=1}^{\lfloor \sqrt{N} \rfloor}e^{2i\pi (k^\delta + \theta k)}\right\rvert+\left\lvert\sum_{k=\lfloor \sqrt{N} \rfloor}^{N}e^{2i\pi (k^\delta + \theta k)}\right\rvert \notag \\
 & \leq \sqrt{N}+ \left(  \lvert \delta N^{\delta-1}-\delta N^{\frac{\delta-1}{2}} \rvert +2\right)\left( \frac{C}{N^{\frac{\delta-2}{2}}}+3 \right) \notag \\
 &\leq C N^{1-\frac{\delta}{2}}.
\end{align*}
This proves Equation \eqref{Exk}. Now, according to Theorem \ref{Th2}, for any contraction $T$ on $L^2(\mu)$, any $f \in L^2(\mu)$ and for $H>\frac{1}{2}$, we have
\[ \lim\limits_{N\rightarrow\infty}\frac{1}{N^{1-\frac{\delta}{4}}\log^H N}\sum_{k=1}^{N} e^{2i\pi k^\delta} T^{k}(f)=0 \quad \mu-a.e..\]
Moreover, by Corollary \ref{cor2}, the series \[\sum_{k\geq 1}\frac{e^{2i\pi k^\delta }}{k^{1-\frac{\delta}{4}+\epsilon}} T^{k}(f)\] exists  $\mu$-a.e., for some $\epsilon>0$.
Notice that our results also improve \cite{Krause}.

As a remark, if we take $w_k=e^{2i\pi \log^\delta k}$ for some $\delta>2$. Similarly to the proof of \eqref{Exk}, a simple calculation yields
\begin{align*}
\sup_{\theta \in \mathbb{R}} \left\lvert\sum_{k=1}^{N}e^{2i\pi (\log^\delta k+ \theta k)}\right\rvert \leq C\frac{N}{\log^{\frac{\delta-1}{2}}N}.    
\end{align*}
Then we can apply Theorem \ref{Th3} and Corollary \ref{cor3}.
\end{Ex}
\begin{Ex}
To deal with
this example, we first state the following inequality  due to Hlawka  (see Theorem 1 in \cite{Hlawka} with $a(t)=\frac{1}{t}$).
\[\sup_{\theta \in \mathbb{R}}\left\lvert\sum_{k=1}^{N}\frac{e^{2i\pi (h\log k+\theta k)}}{k}\right\rvert \leq 30\left(|h|+\frac{1}{|h|}\right),\]
with $h\neq 0$.

Now, according to Theorem \ref{Th5} with $\alpha=0$, for $H>\frac{1}{2}$,
\[
  \lim\limits_{N\rightarrow\infty}\frac{1}{\log^H N}\sum_{k=1}^{N} \frac{e^{2i\pi h\log k}}{k} T^{k}(f) =0\quad \mu-a.e.\] and
  \[\left\lvert\left\lvert \sup_{N>1} \left\lvert\frac{1}{\log^H N}\sum_{k=1}^{N} \frac{e^{2i\pi h\log k}}{k} T^{k}(f) \right\rvert \,\right\rvert\right\rvert_{2,\mu} \leq C ||f||_{2,\mu}.
  \]
  Moreover, according to Remark 7 (assertion 2. with $\alpha=0$), for $H>\frac{1}{2}$, we also have
  \[
  \sum_{k>1} \frac{e^{2i\pi h\log k}}{k\log^H k} T^{k}(f) \quad \text{exists}\quad \mu-a.e.\quad \text{and} \quad
  \left\lvert\left\lvert \sup_{N>1} \left\lvert\sum_{k=2}^{N} \frac{e^{2i\pi h\log k}}{k\log^H k} T^{k}(f) \right\rvert \,\right\rvert\right\rvert_{2,\mu} \leq C ||f||_{2,\mu},
  \]
  respectively.
\end{Ex}
\begin{Ex}
Here we deal with an example which was considered in  \cite{Cohen_2006,Weber2000}.
 Let $(X_k)$ be a sequence of i.i.d random variables defined on some probability space $(\Omega,\mathcal{B}, \PP)$, and let $W_k(\omega)=e^{2i \pi X_k(\omega)}$, $\omega \in \Omega$. Without loss of generality, assume that $\mathbb{E}(W_k)=0$. Let $(u_k)$ be an increasing sequence of integers such that $u_k=O\left(e^{\frac{k}{\log^{2\beta}k}}\right)$ for some $\beta>0$.

According to Theorem 1.1 in \cite{Cohen_2006} in the one-dimensional setting, there exists some random variable $C(\omega)$ which is almost surely (a.s.) finite and independent of $M,N$ and $\theta$ such that, for almost every $\omega \in \Omega$,
\begin{align*}
  \sup_{\theta \in \mathbb{R}}\left\lvert \sum_{k=M+1}^N W_k(\omega) e^{2i\pi \theta u_k}\right\rvert &\leq C(\omega) \sqrt{\sum_{k=M+1}^N\lvert W_k \rvert^2}\times \sqrt{\log u_N} \notag   \\
  &\leq C(\omega) \sqrt{N-M}\times \frac{\sqrt{N}}{\log^\beta N}.
\end{align*}
Theorem \ref{Th3} applied to $\beta>1$ ensures that there exists a measurable set $\Omega^*\subset \Omega$ of full measure, such that for any $\omega \in \Omega^*$, any contraction $T$ on $L^2(\mu)$ and any $f \in L^2(\mu)$,  
\begin{equation}{\label{weber}}
 \lim\limits_{N\rightarrow\infty}\frac{1}{N}\sum_{k=0}^{N-1} W_k(\omega) T^{u_k}(f)=0 \quad \mu-a.e. ~ \text{and} ~~
    \left\lvert\left\lvert\sup_{N\geq1}\left\lvert\frac{1}{N}\sum_{k=0}^{N-1} W_k(\omega) T^{u_k}(f)\right\rvert \, \right\rvert\right\rvert_{2,\mu}\leq C(\omega)\left|\left|f\right|\right|_{2,\mu}.    
\end{equation}
Moreover, by Corollary \ref{cor3}, we also obtain that
\begin{align}{\label{cohen}}
    \sum_{k\geq 1} \frac{W_k(\omega)}{k} T^{u_k}(f) \quad \text{exists} \quad \mu-a.e..
\end{align}
A weak version of \eqref{weber} could be derived from a theorem due to Weber. Indeed, according to Theorem 4.2 in \cite{Weber2000} with $\phi(x)=\frac{x}{\log^{2\beta}x}$, for any $\tau>\frac{3}{2}$, there exists a measurable set $\Omega^*\subset \Omega$ of full measure, such that for any $\omega \in \Omega^*$, any contraction $T$ on $L^2(\mu)$ and any $f \in L^2(\mu)$,
\[
\lim\limits_{N\rightarrow\infty}\frac{1}{N\log^{\tau-\beta} N}\sum_{k=0}^{N-1} W_k(\omega) T^{u_k}(f)=0 \quad \mu-a.e.,\] and
    \[\left|\left|\sup_{N>1}\left|\frac{1}{N\log^{\tau-\beta} N}\sum_{k=0}^{N-1} W_k(\omega) T^{u_k}(f)\right| \, \right|\right|_{2,\mu}\leq C(\omega) \left|\left|f\right|\right|_{2,\mu}.
\]
However, \eqref{weber} is more precise.

Remark that, Cohen and Cuny could also obtain \eqref{cohen} with a more restrictive condition. Indeed, \eqref{cohen} can be deduced from Theorem 1.2 in \cite{Cohen_2006} provided that  
\[\sum_{k\geq 1} \frac{||W_k||_2^2}{k^2} (\log^2k) (\log u_k)<\infty.\]
However this condition needs that
$\sum_{k\geq 1} \frac{1}{k \log^{2\beta-2}k}<\infty$
which holds when $\beta>\frac{3}{2}$. Our result is more general since we only assume that $\beta>1$.
\end{Ex}
\begin{Ex}
Let  $W_k(\omega)=e^{2i \pi X_k(\omega)}$, where $(X_k)$ is an i.i.d sequence of random variables on $(\Omega,\mathcal{B}, \PP)$ and let $u_k=O\left(e^{\frac{\log^2 k}{\log^{2\beta} \log k}}\right)$. According to Theorem 1.1 in \cite{Cohen_2006} in the one-dimensional setting, there exists some random variable $C(\omega)$ which is $\PP-$a.s. finite and independent of $M,N$ and $\theta$ such that, for almost every $\omega \in \Omega$,
\begin{align*}
  \sup_{\theta \in \mathbb{R}}\left\lvert \sum_{k=1}^N \frac{W_k(\omega)}{k} e^{2i\pi \theta u_k}\right\rvert &\leq C(\omega) \sqrt{\sum_{k=1}^N\left\lvert  \frac{W_k}{k} \right\rvert^2}\times \sqrt{\log u_N} \notag   \\
  &\leq C(\omega)  \frac{\log N}{\log^\beta \log N}.
\end{align*}
Applying Theorem \ref{Prop2} with $\beta>1$, there exists a measurable set $\Omega^*\subset \Omega$ of full measure with the following property: for any $\omega \in \Omega^*$, there exists  $C(\omega)<\infty$ such that for any contraction $T$ on $L^2(\mu)$ and any $f \in L^2(\mu)$,
\[\lim\limits_{N\rightarrow\infty}\frac{1}{\log N}\sum_{k=1}^{N} \frac{W_k(\omega)}{k}T^{u_k}(f)=0 \quad \mu-a.e.,\]
and
    \[\left\lvert\left\lvert\sup_{N>1}\left\lvert\frac{1}{\log N}\sum_{k=1}^{N} \frac{W_k(\omega)}{k} T^{u_k}(f)\right\rvert \, \right\rvert\right\rvert_{2,\mu}\leq C(\omega)\left|\left|f\right|\right|_{2,\mu}.\]
The study of weighted means of the form \[\frac{1}{A_N}\sum_{k=1}^N Z_k(\omega)T^k(f),\]
where  $(Z_k)$ is a sequence of random variables and where $(A_k)$ is a sequence of integers, is classical. For example in \cite{bouhkari2002}, it is proved that this mean converges to $0$  $\mu$-a.e. when $A_N=\sqrt{N}\log^\beta N$ ($\beta>2$) and when $(Z_k)$ is an i.i.d sequence of symmetric and square integrable random variables. In \cite{ASSANI1998139}, the same type of result was established when $A_N=N$ and when $(Z_k)$ is a sequence of i.i.d symmetric random variables such that $\mathbb{E}(|Z_k|^p)<\infty$ for some $1<p<\infty$.
\end{Ex}
\begin{Ex}
Let $(X_k)$ be a sequence of independent Bernoulli random variables on some probability space $(\Omega, \mathcal{B}, \mathbb{P})$ such that $\PP(X_k=1)=\frac{1}{\log k}=1-\PP(X_k=0)$. In this example, we consider the random set of integers $\mathcal{N}(\omega):=\{k\geq 2: X_k(\omega)=1\}$. This set is referred to as the Cramer's random model of primes. In particular, we have $\mathcal{N}(\omega)=\{u_k(\omega),  k \geq 1\},$
where \[u_1(\omega)=\inf\{i\geq 2: X_i(\omega)=1 \} \quad \text{and}  \quad u_{k+1}(\omega)=\inf\{i>u_k(\omega): X_i(\omega)=1 \}.\]
In what follows we denote by \[\Pi(N)=\#\mathcal{N}\cap [2,N]=\sum_{k=2}^N X_k.\]
The above random variable satisfies the following law of large numbers:  $\lim\limits_{N\rightarrow\infty}\frac{ \Pi(N)\log N}{N}=1$ (see Proposition 6.2. in \cite{fandominique}). We deal with below the almost everywhere convergence of the following average
\[ \frac{1}{N^\beta} \sum_{k=1}^{N}T^{u_k}(f)= \frac{1}{\Pi(u_N)^\beta} \sum_{k=2}^{u_N}X_k T^{k}(f) \qquad  \left(\frac{1}{2}<\beta\leq1\right).\]
To do it, it is sufficient to deal with the almost everywhere convergence of
\[ \frac{1}{\Pi(N)^\beta} \sum_{k=2}^{N}X_k T^{k}(f).\]
First we write
\begin{align*}
 \frac{1}{\Pi(N)^\beta} \sum_{k=2}^{N}X_k T^{k}(f)=  \frac{1}{\Pi(N)^\beta} \sum_{k=2}^{N}\left(X_k-\frac{1}{\log k} \right)T^{k}(f) +  \frac{1}{\Pi(N)^\beta} \sum_{k=2}^{N} \frac{T^{k}(f)}{\log k}.
\end{align*}
According to Theorem 1.1 in \cite{Cohen_2006} in the one-dimensional setting, there exists some random variable $C(\omega)$ which is $\PP-$a.s. finite and independent of $M,N$ and $\theta$ such that, almost surely
\begin{align*}
 \sup_{\theta \in \mathbb{R}}\left\lvert \sum_{k=M}^N \left(X_k-\frac{1}{\log k} \right) e^{2i\pi \theta k}\right\rvert &\leq C(\omega) \sqrt{\sum_{k=M}^N\left\lvert \left(X_k-\frac{1}{\log k} \right) \right\rvert^2}\times \sqrt{\log N} \notag   \\
  &\leq C(\omega)\sqrt{N-M}\times \sqrt{\log N}.  
\end{align*}
Applying Theorem \ref{Th1}, for $H>2$, we get $\PP-$a.s.
\begin{align*}
  \lim\limits_{N\rightarrow\infty}  \frac{1}{\sqrt{N}\log^H N} \sum_{k=2}^{N}\left(X_k-\frac{1}{\log k} \right)T^{k}(f)=0 \quad \mu-a.e..
\end{align*}
Since $\sup_{N>1}\frac{\sqrt{N}\log^H N}{\Pi(N)^\beta}$ is bounded, we have $\PP-$a.s.
\begin{align*}
  \lim\limits_{N\rightarrow\infty}  \frac{1}{\Pi(N)^\beta} \sum_{k=2}^{N}\left(X_k-\frac{1}{\log k} \right)T^{k}(f)=0 \quad \mu-a.e..  
\end{align*}
Now, by Abel's summation formula,  
\begin{align}{\label{eqex7}}
  \left\lvert\frac{1}{\Pi(N)^\beta} \sum_{k=2}^{N} \frac{T^{k}(f)}{\log k}\right\rvert&\leq \frac{1}{\Pi(N)^\beta} \left(\sum_{k=2}^{N-1} \frac{\log(k+1)-\log k}{\log(k+1)\log k}\left\lvert \sum_{j=2}^kT^{j}(f)\right\rvert+\frac{1}{ \log N}\left\lvert\sum_{k=2}^{N}T^{k}(f)\right\rvert\right) \notag \\
 &\leq \frac{\log^\beta N}{N^\beta}\sum_{k=2}^{N-1} \frac{1}{k^{1-\beta}\log^2 k} \left\vert\frac{1}{k^\beta}\sum_{j=2}^k T^{j}(f)\right\rvert+\frac{1}{ \log^{1-\beta} N}\left\lvert \frac{1}{N^\beta}\sum_{k=2}^{N}T^{k}(f)\right\rvert.
\end{align}
In what follows, we only deal with the case where the function $f$ is such that
\begin{align}{\label{eq4}}
  \lim\limits_{N\rightarrow\infty}\frac{1}{N^\beta} \sum_{k=1}^{N}T^{k}(f)=0 \quad \mu-a.e..
\end{align}
Examples of function $f$ satisfying the above equation are given in \cite{cuny2009,Derriennic,Weber2000}.

 According to the above assumption, the second term of \eqref{eqex7} converges to $0$ $\mu$-a.e.. To deal with the first term of \eqref{eqex7}, recall that
\begin{align*}
    \sum_{k\geq 1}b_k=+\infty \quad \text{and} \quad  \lim\limits_{k\rightarrow\infty}a_k=0 \implies  \lim\limits_{N\rightarrow\infty} \frac{\sum_{k=1}^N b_k a_k}{\sum_{k=1}^N b_k}=0
\end{align*}
for any sequences $(a_k)$, $(b_k)$. Applying this to
 $a_k=\frac{1}{k^\beta}\sum_{j=2}^k T^{j}(f)$ and $b_k=\frac{1}{k^{1-\beta}\log^2 k}$, and using the fact that $\frac{\log^\beta N}{N^\beta}\sum_{k=2}^N\frac{1}{k^{1-\beta}\log^2 k}$ is bounded, it follows that
 \[\lim_{N\to \infty}\frac{\log^\beta N}{N^\beta}\sum_{k=2}^{N-1} \frac{1}{k^{1-\beta}\log^2 k} \left\vert\frac{1}{k^\beta}\sum_{j=2}^k T^{j}(f)\right\rvert=0 \quad \mu-a.e..\]\\
Consequently, $\PP-$a.s.
\[\lim\limits_{N\rightarrow\infty} \frac{1}{N^\beta} \sum_{k=1}^{N}T^{u_k}(f)=0 \quad \mu-a.e..\]
As an open question, is it true that \eqref{eq4} implies that  
\[ \lim\limits_{N\rightarrow\infty}\frac{1}{N^\beta} \sum_{k=1}^{N}T^{p_k}(f)=0 \quad \mu-a.e. \qquad  \left(\frac{1}{2}<\beta\leq1\right),\]
where $\mathcal{P}=\{p_k\}_{k\geq1}$ denotes the set of prime numbers?
\end{Ex}
\bigskip
\noindent\textbf{Acknowledgments} We would like to thank Nicolas Chenavier for his careful reading and for his suggestions.

\end{document}